\documentclass[letterpaper, 10 pt, conference]{ieeeconf}
\IEEEoverridecommandlockouts
 \usepackage{comment}
\usepackage[utf8]{inputenc} % Permet la gestion correcte des accents
\usepackage{graphics} % for pdf, bitmapped graphics files
\usepackage{epsfig} % for postscript graphics files
\usepackage{times} % assumes new font selection scheme installed
\usepackage{amsmath} % assumes amsmath package installed
\usepackage{amssymb} % assumes amsmath package installed
\usepackage{algorithm}
\usepackage{algorithmic}

\let\oldwidebar\widebar  % Sauvegarde de la commande originale
\let \oldwidetilde\widetilde
\let \oldsucc\succ
\let \oldsucceq\succeq
\let \oldprec\prec
\let \oldpreceq\preceq
\let \oldemptyset \emptyset
\let \oldsubset \subset
\let \oldnot \not
\let \oldinfty\infty

\usepackage[mathx]{mathabx}     % Chargement du package %required for widecheck
\let\widebar\oldwidebar  % Restauration de la commande originale
\let\widetilde\oldwidetilde
\let \succeq \oldsucceq
\let \succ \oldsucc
\let \preceq \oldpreceq
\let \prec \oldprec
\let \emptyset \oldemptyset
\let \subset \oldsubset
\let \not \oldnot
\let \infty \oldinfty 

\usepackage[english]{babel} % Specify a different language here -- english by default
\usepackage{mathtools}
\usepackage{optidef}
\usepackage{cite}
\usepackage{mathtools}

\usepackage{todonotes}

\usepackage{lipsum} % Required to insert dummy text. To be removed otherwise.

\makeatletter
\let\NAT@parse\undefined
\makeatother
\usepackage[colorlinks=true,
     urlcolor=black,    
     citecolor=black,
     menucolor=black,    
     linkcolor=black,
     breaklinks=true,
]{hyperref}

\makeatletter
\newcommand*\rel@kern[1]{\kern#1\dimexpr\macc@kerna}
\newcommand*\widebar[1]{%
  \begingroup
  \def\mathaccent##1##2{%
    \rel@kern{0.8}%
    \overline{\rel@kern{-0.8}\macc@nucleus\rel@kern{0.2}}%
    \rel@kern{-0.2}%
  }%
  \macc@depth\@ne
  \let\math@bgroup\@empty \let\math@egroup\macc@set@skewchar
  \mathsurround\z@ \frozen@everymath{\mathgroup\macc@group\relax}%
  \macc@set@skewchar\relax
  \let\mathaccentV\macc@nested@a
  \macc@nested@a\relax111{#1}%
  \endgroup
}
\makeatother
\makeatletter
\renewcommand{\thefigure}{\arabic{figure}}
\makeatother
\title{\LARGE \bf Discounted LQR: stabilizing (near-)optimal state-feedback laws} 

\author{Jonathan de Brusse$^{1,2}$, Jamal Daafouz$^{1,3}$, Mathieu Granzotto$^2$, Romain Postoyan$^1$ and Dragan Ne\v si\'c$^{2}$  % <-this % stops a space
\thanks{*Work funded by the ANR under grant OLYMPIA ANR-23-CE48-0006 and the Discovery Project ARC DP250100300.}% <-this % stops a space
\thanks{$^{1}$ J. de Brusse, J. Daafouz and R. Postoyan are with the Universit\'e de Lorraine, CNRS, CRAN, F-54000, Nancy, France. (emails: \{jonathan.de-brusse, jamal.daafouz,romain.postoyan\}@univ-lorraine.fr). }%
\thanks{$^{2}$ J. de Brusse, D. Ne\v si\'c and M. Granzotto are with the Department of Electrical and Electronic Engineering, University of Melbourne, Parkville, VIC 3010, Australia (emails:
jonathan.debrusse@student.unimelb.edu.au, \{dnesic, mathieu.granzotto\}@unimelb.edu.au).}%
\thanks{$^{3}$ J. Daafouz is also with  Institut Universitaire de France (IUF), F-75000, Paris, France.}
}% <-this % stops a space
\newcommand{\norm}[1]{\left \Vert #1 \right \Vert}

\newcommand{\R}{\mathbb{R}}
\newcommand{\C}{\mathbb{C}}

\newcommand{\Z}{\mathbb{Z}}
\newcommand{\Rlp}{\ensuremath{\mathbb{R}_{>0}}}

\newcommand{\1}{{\bf 1}}
\newcommand{\0}{{\bf 0}}

\newtheorem{thm}{Theorem}

\newtheorem{Example}{Example}
\newtheorem{revexample}{Example 1 revisited}

\newtheorem{rem}{Remark}

\newtheorem{sass}{Standing Assumption}

\begin{document}

\maketitle
\thispagestyle{empty}
\pagestyle{empty}

%%%%%%%%%%%%%%%%%%%%%%%%%%%%%%%%%%%%%%%%%%%%%%%%%%%%%%%%%%%%%%%%%%%%%%%%%%%%%%%%
\begin{abstract} We study deterministic, discrete linear time-invariant systems with infinite-horizon discounted quadratic cost. It is well-known that standard stabilizability and detectability properties are not enough in general to conclude stability properties for the system in closed-loop with the optimal controller when the discount factor is  small. In this context, we first review some of the stability conditions based on the optimal value function found in the learning and control literature and  highlight their conservatism. We then propose novel (necessary and) sufficient conditions, still based on the optimal value function, under which stability of the origin for the optimal closed-loop system is guaranteed. Afterwards, we focus on the scenario where the optimal feedback law is not stabilizing because of the discount factor and the goal is to design an alternative stabilizing near-optimal static state-feedback law. We present both linear matrix inequality-based conditions and a variant of policy iteration to construct such stabilizing near-optimal  controllers. The methods are illustrated via numerical examples.

% We first establish Lyapunov-based conditions that ensure the stability of the optimal solution to the discounted LQR for a given $\gamma$. We then design near-optimal stabilizing state-feedback laws either by solving LMIs or using a near-optimal variant of policy iteration.
\end{abstract}
%\begin{keywords} 
%\end{keywords}
%%%%%%%%%%%%%%%%%%%%%%%%%%%%%%%%%%%%%%%%%%%%%%%%%%%%%%%%%%%%%%%%%%%%%%%%%%%%%%%%

\section{Introduction}\label{sect:introduction}

%The Linear Quadratic Regulator (LQR) is a well-known problem in optimal control. 

While the stability properties of the undiscounted linear quadratic regulator (LQR) problem are well-understood, see e.g., \cite{Anderson-Moore-book,Bertsekas-book12(adp),Kucera72b,book-lewis-12}, the case where the cost function involves a discount factor still presents challenges. Even when a linear time-invariant system is stabilizable and satisfies a standard detectability property with respect to the stage cost, the optimal closed-loop system, i.e., the system in closed-loop with the optimal feedback law, may be unstable because of the discount factor \cite[Example 1]{Postoyan17}, see also \cite{GAIT2018} and \cite{granzotto-tac24(stochastic)} for  stochastic results. It is shown in \cite[Corollary 3]{Postoyan17} that stability is always guaranteed in this context provided that the discount factor $\gamma$ is sufficiently close to $1$. When this is not the case, namely when $\gamma$ is too small, a common technique consists in iteratively increasing the discount factor so that it is eventually close or equal to $1$, see, e.g., \cite{perdomo2021,Agazzi2020,Lamperski_2020,zhao2021}. 
Still, there are scenarios where the discount factor should be fixed and not necessarily close to $1$ because of the problem at hand, or because having a small discount factor allows working with limited amount of data in a learning setting \cite{Tyler22}, and yet ensuring the stability of the closed-loop system is essential. Two questions then naturally arise:
\begin{enumerate}
\item[1)] When is stability guaranteed for the discounted LQR problem for a given value of $\gamma$, possibly small?
\item[2)] If the optimal feedback law is not stabilizing, how to construct a near-optimal stabilizing state-feedback law instead?
\end{enumerate}
The objective of this work is to provide answers to both questions for deterministic, discrete linear time-invariant systems.

Question 1) admits a simple answer: it suffices to evaluate the eigenvalues of the optimal closed-loop system state matrix. However, it is  important to also have at our disposal stability  conditions involving the optimal value function. This is particularly relevant in a learning context where we seek to approximate or learn the optimal value function and use it to certify stability \cite{Tyler22,LAI2023,zhao2021}. We show in this work that the existing conditions are subject to conservatism in the sense that they unnecessarily require $\gamma$ to be close to $1$. To overcome these limitations, we present a novel relaxed sufficient condition that captures the one in \cite{Tyler22, LAI2023} as a special case, thereby reducing conservatism in the range of $\gamma$ values that satisfy it. We then show that a necessary and sufficient condition can be obtained if we add a  quadratic term to the optimal value function to obtain a Lyapunov function for the optimal closed-loop system. 
% To overcome these limitations, we present novel relaxed sufficient conditions under which the optimal value function is indeed a Lyapunov function for the optimal closed-loop system \jonathan{Cette phrase n'est pas correcte, notre condition nous assure seulement qu'il existe une fonction de Lyapunov pour le système optimal mais pas que $V^{\star}_{\gamma}$ est une fonction de Lyapunov}. This novel condition captures the one in \cite{Tyler22, LAI2023} as a special case, thereby reducing conservatism in the range of $\gamma$ values that satisfy it.
% We then show that this condition becomes necessary and sufficient if we add an extra quadratic term to the optimal value function to obtain a Lyapunov function for the optimal closed-loop. \textcolor{red}{Celle-ci n'est pas bonne non plus. Je propose de reformuler tout ca comme suit:  }
Similar techniques have been used in other optimal control problems, see, e.g., \cite{Grimm-et-al-tac2005,Postoyan17,granzotto-tac24(stochastic)}. However, in all these references the extra term used to build a Lyapunov function is positive semi-definite, while the extra term that we propose may not be sign definite, which is key to making our condition both necessary and sufficient. %\textcolor{red}{ je ne sais pas  on doit citer ces deux papiers, je regarde \cite{ZANON2022, GAITSGORY2018} }.

Afterwards, we address question 2).  We derive for this purpose convex optimization problems with LMI constraints, which allow constructing a near-optimal stabilizing state-feedback law. In particular, we present sufficient conditions under which a stabilizing state-feedback can be constructed and whose cost is guaranteed to be less than a given constant for a given initial condition, like in e.g., \cite[Chapters 7 and 10]{Boyd94} and more recently in a data-driven context in \cite{van-waarde-mesbahi-ifac2020data} where similar near-optimal controller designs can be found for undiscounted costs. We further provide LMI conditions, which allow synthesizing a stabilizing state-feedback law that is independent of the considered initial condition this time and with a guaranteed margin on the mismatch with the optimal, but not stabilizing, gain. We finally provide a variant of policy iteration (PI) algorithm that enforces stability at each iteration while ensuring cost improvement. Interestingly, this new algorithm shares similarities with policy gradient methods \cite{Fazel_2018} as we show. All the results are illustrated via several numerical examples for which  closed-loop system stability is ensured  without significantly degrading performance with respect to the  optimal one.

The remainder of the paper is structured as follows. Section \ref{sect:notation} introduces the used notation. Section \ref{secLQRopt} formally states the problem. In Section \ref{sec: Lyap analysis}, we present novel optimal value function-based conditions ensuring  stability for discounted LQR. Section \ref{sect:near-optimal} focuses on the design of near-optimal stabilizing state-feedback laws. Conclusions and  perspectives are discussed in Section \ref{sect:conclusions}.

%%%%%%%%%%%%%%%%%%%%%%%%%%%%%%%%%%%%%%%
\section{Notation}\label{sect:notation}
%%%%%%%%%%%%%%%%%%%%%%%%%%%%%%%%%%%%%%%

Let $\C$ be the set of complex numbers, $\R$ the set of real numbers, $\Z_{\ge0}$ the set of non-negative integers including $0$ and $\Z_{>0}$ the set of strictly positive integers. The symbol $\R^{n\times m}$ stands for the set of real matrices with $n\in \Z_{>0}$ rows and $m\in \Z_{>0}$ columns. For a symmetric matrix $Q\in \R^{n\times n}$, we write $Q\succ 0$ ($Q\succeq0$) if $Q$ is positive (semi-)definite. The identity matrix is denoted by \1 and the zero matrix by \0, whose dimensions depend on the context. The symbol  $(\bullet)$ in a matrix stands for the symmetric term, i.e., $\left[\begin{smallmatrix}
    A & B\\
    B^{\top} & C
\end{smallmatrix}\right]=\left[\begin{smallmatrix}
    A & (\bullet)^{\top}\\
    B^{\top} & C
\end{smallmatrix}\right]$. For any $M\in \R^{n \times m}$ with $n,m\in \Z_{>0}$, the Frobenius norm of $M$ is denoted by $\norm{M}$. For any square matrix $M\in \R^{n\times n}$ with $n\in \Z_{>0}$, the trace of $M$ is denoted by $\text{trace}(M)$ and the spectral radius of $M$ by $\rho(M)$. For any $\lambda\in \C$, $\lambda^*$ denotes its complex conjugate. For an infinite sequence $\boldsymbol{u}:=(u_0,u_1,\dots)$ where $u_0,u_1,...\in \R^{m}$ with $m\in \Z_{>0}$, $\boldsymbol{u}|_k$ stands for the truncation of $\boldsymbol{u}$ to the first $k\in \Z_{> 0}$ elements, i.e., $\boldsymbol{u}|_k:=\big(u_0,\dots,u_{k-1}\big)$ and we use the convention $\boldsymbol{u}|_0=\emptyset$ when $k=0$. 

%%%%%%%%%%%%%%%%%%%%%%%%%%%%%%%%%%%%%%%%%%%%
\section{Problem statement}\label{secLQRopt}
%%%%%%%%%%%%%%%%%%%%%%%%%%%%%%%%%%%%%%%%%%%%

\subsection{Discounted LQR}
Consider the deterministic, discrete-time linear time-invariant system
\begin{equation}
{x}_{k+1} = Ax_k + Bu_{k},
	\label{eq:sysd}
\end{equation}
where $x_k \in \R^n$ and $u_k\in \R^m$ are the state and control input at time $k\in \Z_{\ge0}$, respectively, with $n,m\in\Z_{>0}$.  Matrices $A$ and $B$ are real, of appropriate dimensions and such that the next assumption holds.

\begin{sass}[SA\ref{sass:stabilizability}]\label{sass:stabilizability} The pair $(A,B)$ is stabilizable.\hfill $\Box$
\end{sass}

Necessary and sufficient conditions under which  SA\ref{sass:stabilizability} holds are provided in \cite[Lecture 14]{Hespanha-book-2009}. The discounted (LQR) problem consists of determining, for any given initial state $x_0 \in \R^n$, an infinite length sequence of inputs $\boldsymbol{u}=(u_k)_{k\in \Z_{\ge0}}\in(\R^{m})^{\Z_{\geq 0}}$ that minimizes the cost function
\begin{equation}
J_{\gamma} (x_0, \boldsymbol{u}) = \sum_{k=0}^{\infty} \gamma ^k \big[ \phi(k,x_0,\boldsymbol{u}|_{k})^{\top} Q\phi(k,x_0,\boldsymbol{u}|_{k}) + u_k^{\top} Ru_k \big]
	\label{eq:costd}
\end{equation}
where the discount factor $\gamma\in[0,1]$ is \emph{fixed} and  $\phi(k,x_0,\boldsymbol{u}|_k)$ denotes the state obtained at time $k\in \Z_{\ge0}$ by applying the sequence of inputs $\boldsymbol{u}|_k$ from the initial state $x_0$ to system (\ref{eq:sysd}). We use the convention $\phi(0,x_0,\boldsymbol{u}|_0)=x_0$. We assume that $Q$ and $R$  satisfy the next assumption.

\begin{sass}[SA\ref{sass:observability}]\label{sass:observability} There exist real matrices $C$ and $D$ such that $Q=C^{\top}C \succ 0$ and $R=D^{\top}D \succ 0$.\hfill $\Box$
\end{sass}

SA\ref{sass:observability} implies that $(A,C)$ is observable. The optimal controller, also called \textit{optimal policy}, which minimizes \eqref{eq:costd} is given by the state feedback law  $h^{\star}_{\gamma}(x)= K_{\gamma} x$ for any $x\in \R^{n}$ (see e.g., \cite[Section 4.2]{Bertsekas-book12(adp)}), where
\begin{equation}
K_{\gamma} := -\gamma  (R+\gamma  B^{\top}P_{\gamma}B)^{-1}B^{\top}P_{\gamma}A
	\label{eq:Kopt}
\end{equation}
and  $P_{\gamma}=P_{\gamma}^\top \succ 0$ is the solution to the discrete-time algebraic Riccati equation 
\begin{equation}\label{eq:Ricc}
P _{\gamma}= \gamma  A^{\top}P_{\gamma}A - \gamma^2  A^{\top}P_{\gamma}B(R+\gamma  B^{\top}P_{\gamma}B)^{-1}B^{\top}P_{\gamma}A+Q.
\end{equation}
The optimal cost at a given initial state $x_0\in \R^{n}$ is then given by 
\begin{equation}
\begin{array}{rlll}
V^{\star}_{\gamma}(x_0):=\displaystyle \min_{\boldsymbol{u}}J_{\gamma}(x_0,\boldsymbol{u})=J_{\gamma}(x_0,K_{\gamma})=x_0^\top P_{\gamma}x_0,
\end{array}\label{eq:optimal-value-function}
\end{equation}
where $J_{\gamma}(x_0,K_{\gamma})$ represents the cost induced by the repeated application of the optimal policy $h_{\gamma}^{\star}(x)=K_{\gamma}x$ at each encountered state to (\ref{eq:sysd}). As shown in \cite[Theorem~8]{Kucera72b}, the stabilizability of the pair $(\sqrt{\gamma}A, \sqrt{\gamma}B)$, which follows from SA\ref{sass:stabilizability},   and the detectability of the pair $(\sqrt{\gamma}A,C)$, which follows from SA\ref{sass:observability}, are necessary and sufficient conditions for the Riccati equation \eqref{eq:Ricc} to have a unique symmetric, positive definite  solution. %These conditions are verified here due to the fact that $(A,B)$ and $(A,C)$ are respectively assumed stabilizable and detectable. 
Moreover, as $Q\succ 0$ by SA\ref{sass:observability} and $V^{\star}_{\gamma}(x)=J_{\gamma}(x,K_{\gamma})\ge \ell(x,K_{\gamma}x)=x^{\top}(Q+K_{\gamma}^{\top}RK_{\gamma})x$ for any $x\in \R^{n}$, we derive that 
\begin{equation}\label{eq:P_gamma-greater-Q}
P_{\gamma}\succeq Q \succ  0.
\end{equation}

In the following, we recall that SA\ref{sass:stabilizability} and SA\ref{sass:observability} alone do not always guarantee the stability of the optimal closed-loop system, i.e., system (\ref{eq:sysd}) with input $u=K_\gamma x$.

\subsection{Challenges with stability}
The next example illustrates that SA\ref{sass:stabilizability} and SA\ref{sass:observability} are not sufficient to guarantee the stability of the optimal closed-loop system for all $\gamma\in [0,1]$. 
\begin{Example}\label{Ex: 1}
We consider $A =\big[\begin{smallmatrix}
    -0.97 & 0 \\
    3.88 & 0.97
\end{smallmatrix}\big]$, $B = \big[\begin{smallmatrix}
    2\\ -1
\end{smallmatrix}\big]$, $Q =\big[\begin{smallmatrix}
    2 & 0 \\ 0 & 3
\end{smallmatrix}\big]$ and $R=5$. We take 
 $C=\left[\begin{smallmatrix}
    \sqrt{2}&0\\
    0& \sqrt{3}\\
\end{smallmatrix}\right]$ and $D=\sqrt{5}$ so that SA\ref{sass:stabilizability} and SA\ref{sass:observability} hold. The evolution of the spectral radius of the matrix $A+BK_{\gamma}$ as a function of $\gamma$  is depicted in Fig.~\ref{fig:enter-label}. As we can observe, the optimal closed-loop system is unstable when $\gamma\in [0.02, 0.12]$. 
More precisely, stability is guaranteed for $\gamma\in[0,0.02)$, before being lost  by increasing $\gamma$ in the interval $[0.02,0.12]$ after which it is recovered for $\gamma\in(0.12,1]$. \hfill $\Box$
\begin{figure}[ht]
    \centering
    \includegraphics[width=8cm, height=5cm]{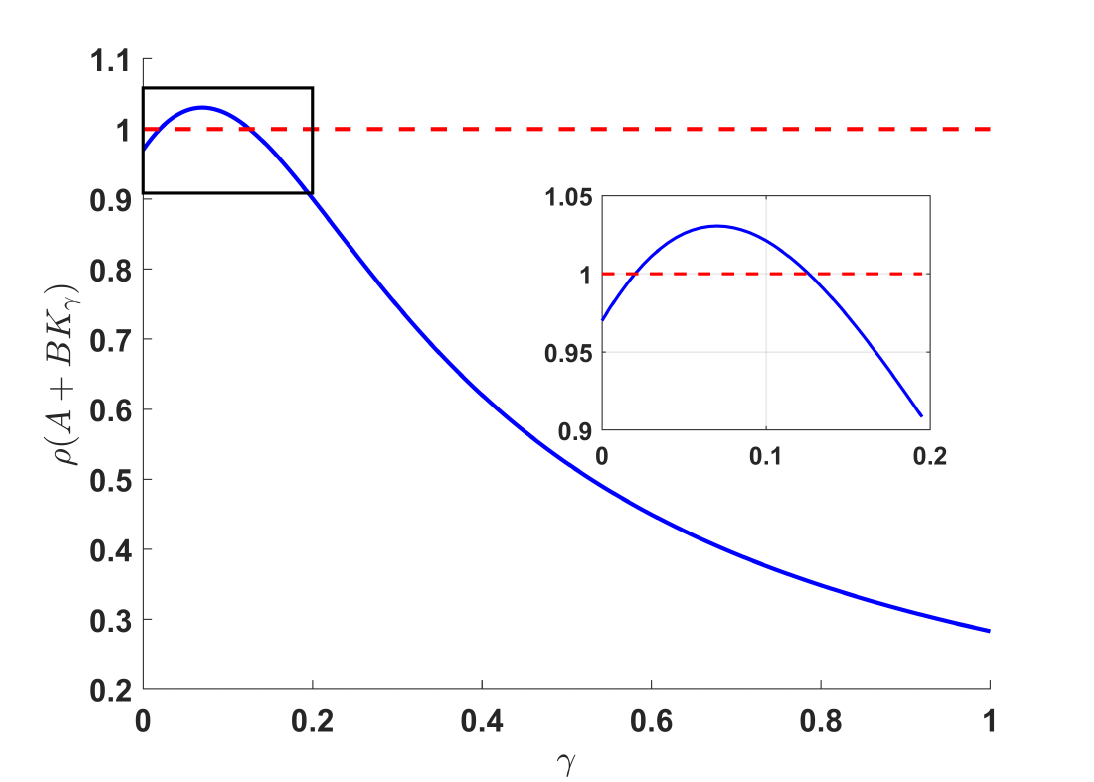}
    \caption{ Spectral radius of $A+BK_{\gamma}$ over the interval $[0,1]$, with a zoom on   $\gamma\in[0,0.2]$ for Example~\ref{Ex: 1}.}
    \label{fig:enter-label}
\end{figure}
\end{Example}

To clarify this phenomenon, we first introduce some relevant notations. Given any $\gamma\in [0,1]$, ${\cal K}_{\gamma}$ denotes the set of all state feedback control gains $K\in\R^{m\times n}$ such that the matrix $ \sqrt{\gamma} (A + BK)$  is Schur, i.e., 
\begin{equation}
\begin{array}{rllll}
\mathcal{K}_{\gamma} & := & \{K\in\R^{m\times n}\,:\,\rho\left(\sqrt{\gamma}(A+BK)\right)<1\}.
\end{array}
\end{equation}
It is important to notice that  $\sqrt{\gamma}(A+BK_\gamma)$ is \emph{not} the state matrix of the optimal closed-loop system, whenever  $\gamma\in[0,1)$. 

As $\gamma\in [0,1]$, we derive that $\mathcal{K}_{1}\subset \mathcal{K}_{\gamma}$, however $\mathcal{K}_{\gamma}\not\subset \mathcal{K}_{1}$ in general. Indeed,  
%Note that if a gain $K\in \mathcal{K}_1$ then for any $\gamma\in [0,1]$, $K\in \mathcal{K}_{\gamma}$ as $\mathcal{K}_{1}\subseteq \mathcal{K}_{\gamma}$. 
 the optimal control gain $K_{\gamma}\in {\cal{K}}_{\gamma}$ and $P_{\gamma}$ are solutions to the Lyapunov equation 
\begin{equation}\label{eq:clsdgam}
\small
(\sqrt{\gamma}(A+BK_{\gamma}))^\top P_{\gamma}(\sqrt{\gamma}(A+BK_{\gamma}))-P_{\gamma}+{K_{\gamma}}^\top R K_{\gamma}+Q =0.
\end{equation}
Thus, the optimal value function $V^{\star}_{\gamma}$ in (\ref{eq:optimal-value-function}) %$:x \mapsto x^\top P_{\gamma} x$ 
is a Lyapunov function for the system
\begin{equation}\label{eq:sys-sqrt-gamma}
x_{k+1} = \sqrt{\gamma}(A+BK_{\gamma})x_k.
\end{equation}
However, there is a priori no reason for $V_{\gamma}^{\star}$ to also be Lyapunov function for the optimal closed-loop system, i.e., system (\ref{eq:sysd}) with $u_k=K_\gamma x_k$, which corresponds to
\begin{equation}\label{eq:sys-optimal-closed-loop}
x_{k+1} = (A+BK_{\gamma})x_k.
\end{equation}
% does not necessarily hold for the dynamics governed by $A+BK_{\gamma}$. Notably, because the optimal policy for the discounted LQR problem generally differs from the optimal policy for the undiscounted case ($\gamma = 1$), and may fail to stabilize system \eqref{eq:sysd}, as illustrated in Example \ref{Ex: 1} and \cite[Example 1]{Postoyan17}. 
In other words, although $K_{\gamma}  \in {\cal K}_{\gamma}$, there is no guarantee that $K_{\gamma} \in {\mathcal K}_1$, which means that the optimal feedback law may not stabilize system (\ref{eq:sysd}) as recalled in Example \ref{Ex: 1}.

\subsection{Objectives}

The first objective of this work is to develop conditions based on the optimal value function that allows assessing the stability of the optimal closed-loop system (\ref{eq:sys-optimal-closed-loop}), thereby providing alternative to the eigenvalue test to answer question 1) in the introduction. This is the purpose of Section \ref{sec: Lyap analysis}.

%based conservative Lyapunov-based conditions using matrix inequalities to assess the stability of the optimal closed-loop system for a given $\gamma$. 
The second objective is to design near-optimal state-feedback laws that ensure the stability of system (\ref{eq:sysd})  when the optimal controller $u=K_\gamma x$ is not stabilizing for system (\ref{eq:sysd}).  This is addressed in Section \ref{sect:near-optimal}.%To achieve this, we explore both convex optimization approaches with LMI-based constraints and a variant of the policy iteration algorithm.

%µThe next two sections address these two object

%%%%%%%%%%%%%%%%%%%%%%%%%%%%%%%%%%%%%%%%%%%%
\section{Stability analysis based on the optimal value function}\label{sec: Lyap analysis}
%%%%%%%%%%%%%%%%%%%%%%%%%%%%%%%%%%%%%%%%%%%%
%In this section, we identify conditions under which $K_{\gamma} \in {\cal K}_1$ \textit{without} altering $\gamma$. 
We begin by revisiting a condition previously proposed in the literature, see, e.g., \cite{Tyler22, LAI2023}, which is stated as follows
\begin{equation}\label{eq:cnsK}
Q+(\gamma -1) P_{\gamma} \succ 0.
\end{equation}
This condition guarantees that the optimal value function $V^{\star}_{\gamma}$ is a \emph{common} Lyapunov function, in the sense that it is a Lyapunov function for both systems (\ref{eq:sys-sqrt-gamma}) and (\ref{eq:sys-optimal-closed-loop}), thereby  ensuring that $K_{\gamma}  \in {\cal K}_{1}$.
Indeed, we have seen above that $K_{\gamma}$ given in (\ref{eq:Kopt}) and $P_{\gamma}$ solution to (\ref{eq:Ricc}) verify \eqref{eq:clsdgam}, which can be written as 
	\begin{equation}\label{eq:clsd}
    \small
	(A+BK_{\gamma})^\top \gamma{P}_{\gamma} (A+BK_{\gamma})-\gamma{P}_{\gamma}+K_{\gamma}^\top R K_{\gamma}+Q+ (\gamma -1) P_{\gamma}=0.
	\end{equation}
Condition \eqref{eq:cnsK} ensures that $ (A+BK_{\gamma})^\top \gamma P_\gamma(A+BK_{\gamma})-\gamma P_\gamma$  is negative definite, thereby making $V^{\star}_{\gamma}$ a common Lyapunov function. % for both systems (\ref{eq:sys-sqrt-gamma}) and (\ref{eq:sys-optimal-closed-loop}). 

We argue that condition (\ref{eq:clsd}) is subject to some conservatism. Indeed, we know that $P_{\gamma} \succeq Q$  by (\ref{eq:P_gamma-greater-Q}) and  condition \eqref{eq:cnsK} is equivalent to $ P_{\gamma}  \prec \frac{1}{1-\gamma}Q$. Thus, $P_\gamma$ is constrained within bounds determined by $Q$, which makes  \eqref{eq:cnsK}  restrictive in general. % It is more likely to hold when $\gamma$ is close to one, but when $\gamma$ decreases, satisfying this condition becomes increasingly difficult,  thereby requiring less conservative conditions. 
The next theorem provides an alternative, relaxed condition, which does not rely on a Lyapunov construction.

\begin{thm}\label{thKuc}
Given any fixed $\gamma \in [0, 1]$, consider the optimal gain $K_{\gamma}$ in (\ref{eq:Kopt}) and the solution $P_{\gamma}$ to \eqref{eq:Ricc}. % ensuring $K_{\gamma} \in {\cal K}_{\gamma}$.  
Then, $K_{\gamma} \in {\cal K}_1$ if the following holds   
\begin{equation}\label{eqKuc}
\gamma ^2P_{\gamma} BR^{-1}B^\top P_{\gamma} + Q + (\gamma -1) P_{\gamma} \succ 0.
\end{equation}
\hfill $\Box$
\end{thm} 

\begin{proof} Let $\gamma\in[0,1]$. We draw inspiration from \cite[Theorem~7]{Kucera72b}, which states that the detectability condition is both necessary and sufficient for the solution of the discrete-time Riccati equation (\ref{eq:Ricc}) to be the unique, symmetric, positive definite solution. We begin by observing that \eqref{eq:Ricc} is equivalent to 
\begin{equation}
\gamma A^\top P_{\gamma} (A - \gamma B(R + \gamma B^{\top}P_{\gamma} B)^{-1}B^\top P_{\gamma} A) = P_{\gamma} - Q 
\end{equation}
which simplifies to, in view of (\ref{eq:Kopt}),
\begin{equation}\label{eqpr1}
\gamma A^{\top}P_{\gamma}(A+BK_{\gamma}) = P_{\gamma} - Q.
\end{equation}
Next, consider the optimal closed-loop state matrix $A+BK_{\gamma}$. By applying standard algebraic manipulations, including the matrix inverse lemma, we derive 
\begin{equation}\label{eqpr2}
A+BK_{\gamma} = A - \gamma B R^{-1}B^\top P_{\gamma}(A+BK_{\gamma}).
\end{equation}
The goal is to prove that $A+BK_\gamma$ is Schur. For this purpose, we proceed by contradiction and assume that there exists an eigenvalue of $A+B K_\gamma$ outside the interior of the unit disk, i.e., there exist $\lambda \in \C $ with $|\lambda|\geq 1$ and $z\in\R^{n}\backslash\{0\}$ such that
\begin{equation}\label{eq:proof-eigenvector}
(A+BK_{\gamma})  z = \lambda z.
\end{equation}
Then by \eqref{eqpr1},
\begin{equation}
\begin{array}{rllll}
\gamma   z^{\top} (A+BK_{\gamma})^\top P_{\gamma} A z & = &  z^{\top} P_{\gamma}z - z^{\top}Q z,
\end{array}\label{eq:proof-suff-cond-stab-eq1}
\end{equation}
and, applying (\ref{eq:proof-eigenvector}), we obtain
\begin{equation}
\begin{array}{rllll}
\gamma \lambda ^*  z^{\top}  P_{\gamma} A z & = & z^{\top} P_{\gamma}z - z^{\top}Q z.
\end{array}\label{eq:proof-contradiction1}
\end{equation}
On the other hand, by  \eqref{eqpr2},
\begin{equation}
\begin{array}{rllll}
    \gamma \lambda ^* z^{\top} P_{\gamma}A z &=& \gamma \lambda ^* z^{\top} P_{\gamma}(A+BK_{\gamma}) z \\
& &+ \gamma ^2 \lambda ^* z^{\top} P_{\gamma}BR^{-1}B^{\top}P_{\gamma} (A+BK_{\gamma}) z,
\end{array}\label{eq:proof-suff-cond-stab-eq2}
\end{equation}
which becomes, by invoking (\ref{eq:proof-eigenvector}) again,
\begin{equation}
\begin{array}{rllll}
\gamma \lambda ^*  z^{\top}  P_{\gamma} A z & = & \gamma \lambda ^* \lambda z^{\top} P_{\gamma}  z \\ 
&& + \gamma ^2 \lambda ^* \lambda z^{\top} P_{\gamma}BR^{-1}B^{\top}P_{\gamma}  z.
\end{array}\label{eq:proof-contradiction2}
\end{equation}
We derive from (\ref{eq:proof-contradiction1}) and (\ref{eq:proof-contradiction2}) that 
\begin{equation}\label{eqzero}
z^{\top} ( \gamma ^2 \lambda ^* \lambda  P_{\gamma} BR^{-1}B^\top P_{\gamma} + Q - P_{\gamma}  + \gamma \lambda ^* \lambda  P_{\gamma} ) z = 0.
\end{equation}
Finally, since $  \lambda ^* \lambda =|\lambda|\geq 1$, we have
\begin{multline}\label{eqimpo}
 \gamma ^2 \lambda ^* \lambda  P_{\gamma} BR^{-1}B^\top P_{\gamma} + Q - P_{\gamma}  + \gamma \lambda ^* \lambda  P_{\gamma}  \\
 \succeq   \gamma ^2P_{\gamma} BR^{-1}B^\top P_{\gamma} + Q + (\gamma -1) P_{\gamma}.
\end{multline}
We deduce from (\ref{eqKuc}) that 
\begin{equation}
\begin{array}{rllll}
 \gamma ^2 \lambda ^* \lambda  P_{\gamma} BR^{-1}B^\top P_{\gamma} + Q - P_{\gamma}  + \gamma \lambda ^* \lambda  P_{\gamma}  
 \succ 0.
\end{array}
\end{equation}
This last inequality contradicts (\ref{eqzero}) as $z\neq 0$. Hence, we have proved that $\rho(A+BK_\gamma)<1$, i.e., $K_{\gamma}\in\mathcal{K}_1$, which concludes the proof.
% If inequality \eqref{eqKuc} is satisfied, the inequality \eqref{eqimpo} prevents \eqref{eqzero} from being satisfied with  $z \neq 0$ which would contradict the fact that $P_{\gamma}$ is the solution to the Riccati equation \eqref{eq:Ricc}. Hence, with  $P_{\gamma}$  solution to \eqref{eq:Ricc}, condition \eqref{eqKuc} is only satisfied if $A+BK_{\gamma}$ is Schur or equivalently $K_{\gamma} \in {\cal K}_1$.     
\end{proof}

Theorem~\ref{thKuc} provides a less conservative condition to ensure that $K_{\gamma} \in {\cal K}_1$ compared to (\ref{eq:cnsK}) as it covers it as a special case. To see it, note that $\gamma ^2P_{\gamma} BR^{-1}B^\top P_{\gamma} \succeq 0$ for any $\gamma\in [0,1]$, which implies that $\gamma ^2P_{\gamma} BR^{-1}B^\top P_{\gamma} + Q + (\gamma -1) P_{\gamma}\succeq Q + (\gamma -1) P_{\gamma}$. Therefore, the range of $\gamma$ values that satisfy \eqref{eqKuc} always includes  the range of $\gamma$ values satisfying \eqref{eq:cnsK}, with both being equal when  $\gamma ^2P_{\gamma} BR^{-1}B^\top P_{\gamma}=0$.

Another sufficient stability condition found in the literature  \cite[(7)]{zhao2021} is given by
\begin{equation}
    \label{eq:LyapKRK}
    K_{\gamma}^{\top}RK_{\gamma}+Q+(\gamma-1)P_{\gamma}\succ0.
\end{equation}
Condition (\ref{eq:LyapKRK})  also guarantees that $V^{\star}_{\gamma}$ is a common Lyapunov function. However, it is unclear whether the satisfaction of (\ref{eq:LyapKRK}) ensures that \eqref{eqKuc} holds. We show below that  (\ref{eq:LyapKRK}) imposes a much smaller range of $\gamma$ than (\ref{eqKuc}) for Example \ref{Ex: 1}. In any case, all the stability conditions presented so far are  sufficient. To obtain necessary and sufficient conditions, the next theorem establishes that the optimal value function $V^{\star}_\gamma$ in (\ref{eq:optimal-value-function}) needs to be modified to become a Lyapunov function by adding a suitable  quadratic term. %However, it remains to be determined whether (\ref{eqKuc}) consistently extends beyond the requirement for a common Lyapunov function \romain{what does it mean ``consistently extend''?? We need to be more careful with the words we are using}. Therefore,  it is desirable to establish a necessary and sufficient condition that holds for a given value of  $\gamma$ without relying on the existence of a common Lyapunov function. The following theorem provides such a result.
\begin{thm}\label{thcns}
Given any fixed $\gamma \in [0, 1]$, consider the optimal gain $K_{\gamma}$ in (\ref{eq:Kopt}) and the solution $P_{\gamma}$ to \eqref{eq:Ricc}. Then, $K_{\gamma} \in {\cal K}_1$ if and only if there exists a symmetric matrix $X \in \R^{n \times n}$ such that $\gamma P_{\gamma} + X \succ 0 $ and 
\begin{equation}\label{eq:condcns}
 K_{\gamma}^\top R K_{\gamma} + Q + (\gamma - 1) P_{\gamma} + X - (A+BK_{\gamma})^\top X(A+BK_{\gamma}) \succ 0.
\end{equation}
\end{thm} $\hfill \Box$

\begin{proof} Let $\gamma\in[0,1]$, 
$K_{\gamma} \in {\cal K}_1$ is equivalent to find $S = S^\top \succ 0$ such that
\begin{equation}\label{eq:lyapX}
(A+BK_{\gamma})^\top S(A+BK_{\gamma})-S \prec 0.
\end{equation}
Let $X:=S-\gamma P_\gamma$, then $S = \gamma P_{\gamma}  +X$. Hence, \eqref{eq:lyapX} is  equivalent to
$(\sqrt{\gamma}(A+BK_{\gamma}))^\top P_{\gamma}(\sqrt{\gamma} (A+BK_{\gamma}))-\gamma P_{\gamma}+(A+BK_{\gamma})^\top X(A+BK_{\gamma}) -X \prec 0$. 
By (\ref{eq:clsdgam}), this last inequality is equivalent to 
\begin{equation}
P_{\gamma}\!-\!{K_{\gamma}}^\top R K_{\gamma}\!-\!Q \!+\!(A+BK_{\gamma})^\top X(A+BK_{\gamma})-\gamma P_{\gamma} -X \prec 0,
\end{equation}
which corresponds to (\ref{eq:condcns}), thereby proving the desired result.\end{proof}

Theorem \ref{thcns} provides a necessary and sufficient condition to ensure the stability of the optimal closed-loop system (\ref{eq:sys-optimal-closed-loop}) for a given value of $\gamma$. %Thus, like the eigenvalue test, the matrix inequality \eqref{eq:condcns} can serve as a stability criterion. 
Note that the requirement that $\gamma P_{\gamma} + X \succ 0$ does not necessarily imply that $X\succeq 0$. %We note that when $X = 0$, we recover the  common Lyapunov function case.%(see \eqref{eq:clsd}).

To conclude this section, we revisit Example \ref{Ex: 1} in the light of the results presented so far. 
\begin{revexample} %Consider the model and the cost given in Example~\ref{Ex: 1}. We begin by determining the range of values for $\gamma$ that guarantee the existence of a common Lyapunov function for the optimal policy. 
The condition %$P_{\gamma}  \prec \frac{1}{1-\gamma}Q$ and 
\eqref{eq:LyapKRK} is satisfied for $\gamma \in [0.97, 1]$, while the new condition \eqref{eqKuc} is feasible for $\gamma \in [0.30, 1]$. On the other hand, the necessary and sufficient condition \eqref{eq:condcns} yields that $\gamma$ needs to be in $[0, 0.02) \cup (0.12, 1]$ for system (\ref{eq:sys-optimal-closed-loop}) to be stable, which is consistent with Fig.~\ref{fig:enter-label}.~\hfill $\Box$ \end{revexample}  %seems cover the entire range where $K_{\gamma} \in {\cal K}_1$.

When $\gamma \in [0.02, 0.12]$, the optimal gain $K_\gamma$ is not stabilizing in Example \ref{Ex: 1}. In this case, we may then be interested in designing an alternative, stabilizing gain whose cost is close to $V^{\star}_{\gamma}$ in (\ref{eq:optimal-value-function}): this is the purpose of the next section. 

%obvious limitation of the optimal approach discussed in Theorem~\ref{thcns} is that if the optimal gain fails to stabilize the system, no solution can be obtained, leading to failure for certain values of~$\gamma$. As seen in the previous example, this issue arises for $\gamma \in [0.02, 0.12]$. In contrast, a near-optimal method would address this problem by guaranteeing stability for any given $\gamma$, even when the optimal method proves ineffective. In the next section, we introduce such a near-optimal alternative.

%%%%%%%%%%%%%%%%%%%%%%%%%%%%%%%%%%%%%%%%%%%%
\section{Stabilizing near-optimal state-feedback laws}
\label{sect:near-optimal}
%%%%%%%%%%%%%%%%%%%%%%%%%%%%%%%%%%%%%%%%%%%%

We are motivated in this section by the scenario where $K_\gamma\notin\mathcal{K}_1$. The objective is to design another gain, either denoted $\widehat K_\gamma$ or $\widebar K_\gamma$ as we will see, such that $A+B\widehat{K}_\gamma$ (or $A+B\widebar{K}_\gamma$) is Schur and that is near-optimal. By near-optimal, we either mean that:
\begin{enumerate}
    \item[(i)] the induced cost $J_{\gamma}(x_0,\widehat K_\gamma)$ is less than a given constant that is minimized  for a given initial condition (Section \ref{subsect:lmi-based-design-guaranteed-cost});
    \item[(ii)] $\widebar{K}_\gamma$ is as close as possible to the optimal gain $K_\gamma$ (Section \ref{subsect:lmi-based-design-close-gain}).
\end{enumerate} 
We finally introduce a variant of PI  that generates a stabilizing gain while improving the induced cost by the previous controller at each iteration (Section \ref{subsect:pi}).
% we first present LMI conditions for constructing a near-optimal gain, ensuring that the induced cost stays below a specified upper bound. We then use these conditions to design stabilizing state-feedback laws through convex optimization problems. A similar approach is also proposed for designing a stabilizing controller that is close to the optimal one. We then introduce a variant of the policy iteration (PI) algorithm that guarantees the generation of a stabilizing gain while improving the induced cost by this controller at each iteration.

\subsection{LMI-based design with guaranteed cost}\label{subsect:lmi-based-design-guaranteed-cost}

Without loss of generality\footnote{Note that if an initial choice of $C$ and $D$ satisfying SA\ref{sass:observability} does not yield $C^{\top}D = \0$, we can always redefine them as $C_2 =\left [\begin{smallmatrix}
    C\\
    \0
\end{smallmatrix}\right ]$ and $D_2 =\left [\begin{smallmatrix}
    \0\\
    D
\end{smallmatrix}\right]$, which ensures that $C_2C_2^{\top} = Q$, $D_2D_2^{\top} = R$, and $C_2^{\top}D_2 = \0$.}, we consider that SA\ref{sass:observability} holds with  $C^\top D = \0$. 
The closed-loop form of \eqref{eq:Ricc}, expressed as \eqref{eq:clsdgam}, can be written as 
\begin{equation}\label{eq:close}
(\sqrt{\gamma}(A+BK_{\gamma}))^\top P_{\gamma}(\sqrt{\gamma}(A+BK_{\gamma}))- P_{\gamma}+ {\cal L}^\top_{K_{\gamma}} {\cal L}_{K_{\gamma}}  = 0
\end{equation}
where ${\cal L}_{K_\gamma} := C+DK_{\gamma}$.
The next theorem presents a method to synthesize near-optimal, stabilizing gains $\widehat{K}_\gamma$, whose associated cost is less than a known constant for a given initial condition, like in \cite[Chapter 10]{Boyd94} and \cite{van-waarde-mesbahi-ifac2020data} where undiscounted costs are considered.

\begin{thm}\label{thmGCGXGG}
Given any fixed $\gamma \in [0, 1]$, $\mu\in\R_{>0}, $ initial condition $x_0\in\R^{n}$, consider cost (\ref{eq:costd}) with $C^\top D = \0$.  Suppose there exist $X_{\gamma}\in \R^{n\times n}$, $Z_{\gamma}\in \R^{n\times n}$ symmetric and  $G_{\gamma}\in \R^{n\times n}$,  $Y_{\gamma}\in \R^{m\times n}$ such that 
%\begin{equation}\label{suffiDARIH2GGi}
\begin{subequations}
  \begin{align}
\left[	\begin{array}{cc}
 				\mu&  (\bullet)^\top \\
 				x_0^\top  & X_{\gamma}
 			\end{array}\right] \succeq 0 \label{suffiDARIH2GGi-a}\\
			\left[	\begin{array}{ccc}
 				G_{\gamma}+G_{\gamma}^\top-X_{\gamma} &  (\bullet)^\top &  (\bullet)^\top\\
 				\sqrt{\gamma }AG_{\gamma}+	\sqrt{\gamma }BY_{\gamma} & X_{\gamma} &  (\bullet)^\top \label{suffiDARIH2GGi-b}\\
 				CG_{\gamma} + DY_{\gamma} & \0& \1
 			\end{array}\right] \succeq 0 \\
			\left[	\begin{array}{cc}
 				G_{\gamma} + G_{\gamma}^\top -Z_{\gamma} &  (\bullet)^\top \\
 				{A}G_{\gamma}+{B}Y_{\gamma} & Z_{\gamma}
 			\end{array}\right] \succ 0.\label{suffiDARIH2GGi-c}
\end{align}\label{suffiDARIH2GGi}
\end{subequations}
Then $\widehat{K}_{\gamma} = Y_{\gamma}G_{\gamma}^{-1}\in\mathcal{K}_1$ and $J_{\gamma}(x_0,\widehat{K}_{\gamma})\leq x_0^\top X^{-1}_{\gamma} x_0 \leq \mu$.  Moreover, if  \eqref{suffiDARIH2GGi} is feasible for some $\gamma$, then it is feasible for all $\gamma' \in [0, \gamma]$. \hfill $\Box$
 \end{thm}
\begin{proof} Let $\gamma\in[0,1]$, $\mu\in\Rlp$ and $x_0\in\R^{n}$. 
By taking the Schur complement of $X_{\gamma}$ in the matrix given in (\ref{suffiDARIH2GGi-a}), we derive that if $ x_0^\top X^{-1}_{\gamma} x_0  \leq \mu$ then (\ref{suffiDARIH2GGi-a}) holds. When verified, the inequality (\ref{suffiDARIH2GGi-c}) ensures that 
\begin{equation}\label{Lyapeq}
(A+B\widehat{K}_{\gamma})^\top Z_{\gamma}^{-1}(A+B\widehat{K}_{\gamma})-Z_{\gamma}^{-1} \prec  0. 
\end{equation}
Indeed, when \eqref{suffiDARIH2GGi-c} is verified, $Z_{\gamma}$ and its inverse are guaranteed to be positive definite matrices and we have
\begin{equation}
(Z_{\gamma}-G_{\gamma})^{\top}Z_{\gamma}^{-1}(Z_{\gamma} -G_{\gamma})\succeq 0,
\end{equation}	
which is equivalent to
\begin{equation}
G_{\gamma}^{\top}Z_{\gamma}^{-1}G_{\gamma} \succeq G_{\gamma} + G_{\gamma}^{\top} -Z_{\gamma}
\end{equation}
As a consequence, 
\begin{equation}
	\left[	\begin{array}{cc}
		G_{\gamma} + G_{\gamma}^{\top} -Z_{\gamma} &  (\bullet)^\top\\
		{A}G_{\gamma}+{B}Y_{\gamma} & Z_{\gamma}
	\end{array}\right] \succ  0
\end{equation}
	implies that
  \begin{equation}\label{ZXeq}
	 	\left[	\begin{array}{cc}
	 		G_{\gamma}^{\top}Z_{\gamma}^{-1}G_{\gamma} &  (\bullet)^\top\\
	 		{A}G_{\gamma}+{B}Y_{\gamma} & Z_{\gamma}
	 	\end{array}\right] \succ  0
	 \end{equation}
	 and hence
	$$
	\left[	\begin{array}{cc}
		Z_{\gamma}^{-1} &  (\bullet)^\top \\
		{A}+{B}Y_{\gamma}G_{\gamma}^{-1} & Z_{\gamma}
	\end{array}\right] \succ  0,
	$$
which corresponds, by taking the Schur complement of $Z_{\gamma}$, to \eqref{Lyapeq}. Since \eqref{Lyapeq} holds, we have that $\widehat K_{\gamma}\in\mathcal{K}_1$. 

Using a similar reasoning, we obtain that when \eqref{suffiDARIH2GGi-b} is verified then
\begin{equation}\label{eq: suboptfinalform}
     \left[	\begin{array}{cc}
 		X_{\gamma}^{-1}-Q-(Y_{\gamma}G_{\gamma}^{-1})^{\top}RY_{\gamma}G_{\gamma}^{-1} & (\bullet)^{\top}\\
        \sqrt{\gamma}(A+BY_{\gamma}G_{\gamma}^{-1})&X_{\gamma}
 			\end{array}\right]\succeq0
\end{equation}
and recalling that $\widehat{K}_{\gamma}=Y_{\gamma}G_{\gamma}^{-1}$, by taking the Schur complement of $X_{\gamma}$, \eqref{eq: suboptfinalform} implies 
\begin{equation}
\label{eq: subopt}
    \sqrt{\gamma}(A\!+\!B\widehat{K}_{\gamma})^{\top}X_{\gamma}^{-1}\sqrt{\gamma}(A+B\widehat{K}_{\gamma})\!-\!X_{\gamma}^{-1}\!+\!Q\!+\!\widehat{K}_{\gamma}R\widehat{K}_{\gamma}\preceq0,
\end{equation}
which guarantees that $J_{\gamma}(x_0,\widehat{K}_{\gamma})\leq x_0^{\top}X_{\gamma}^{-1}x_0$ for any $x_0\in \R^{n}$.

Finally, let $\gamma'\in[0,\gamma]$. Inequality (\ref{suffiDARIH2GGi-b}) is equivalent to
\begin{equation}\label{eq:proof-true-for-gamma-implies-true-for-gamma'}
\left[	\begin{array}{ccc}
	G_{\gamma'}+G_{\gamma'}^\top-X_{\gamma'} &  (\bullet)^\top&  (\bullet)^\top \\
	AG_{\gamma'}+	BY_{\gamma'} & \dfrac{1}{\gamma'}X_{\gamma'}&   (\bullet)^\top \\
	CG_{\gamma'} + DY_{\gamma'} & \0& \1
\end{array}\right] \succeq  0
\end{equation}
As ${X_{\gamma}}/{\gamma' } \succeq {X_{\gamma}}/{\gamma} $, (\ref{eq:proof-true-for-gamma-implies-true-for-gamma'}) is  feasible %for $\gamma \in [0, \widebar{\gamma} ]$ 
by choosing $G_{\gamma'}=G_{\gamma}$, $X_{\gamma'}=X_{\gamma}$ and $Y_{\gamma'}=Y_{\gamma}$. This concludes the proof. 
\end{proof}

%Theorem~\ref{thmGCGXGG} deserves some comments. First, the case where the stability condition \eqref{Lyapeq} holds with a common Lyapunov function  $\widetilde{P}_{\gamma} = \widecheck{P}_{\gamma}$ is recovered, as a particular case by setting $Z_{\gamma}=X_{\gamma}$. Second, imposing $G_{\gamma} = X_{\gamma}$, directly links $\widehat{K}_{\gamma} $
%to the Lyapunov matrix $\widecheck{P}_{\gamma}$, which may introduce conservatism, even though the resulting near-optimal gain does not necessarily depend on the existence of a common Lyapunov function. 
%In conclusion, Theorem~\ref{thmGCGXGG}  provides a near-optimal solution with minimal conservatism.  It can be used to compute a state-feedback gain  $\widehat{K}_{\gamma}$ that belongs to ${\cal K}_1$, while minimizing $\mu$. 

Theorem \ref{thmGCGXGG} allows, when (\ref{suffiDARIH2GGi}) is  feasible, to construct a near-optimal, stabilizing gain $\widehat{K}_\gamma$ for system (\ref{eq:sysd}) where, by near-optimal, we mean that, for a given initial condition $x_0\in \R^{n}$,\begin{equation}\label{eq;Guarcost}
    x_0^{\top}P_{\gamma}x_0\leq J_{\gamma}(x_0,\widehat{K}_\gamma)\leq  x_0^{\top}X^{-1}_{\gamma}x_0 \leq \mu.
\end{equation} Note that a natural objective would be to minimize the quadratic form $x^{\top}_0X^{-1}_{\gamma}x_0$ only under the constraints \eqref{suffiDARIH2GGi-b} and \eqref{suffiDARIH2GGi-c}. However, optimization toolboxes such as Matlab and Yalmip cannot handle the explicit inverse of a matrix variable. To overcome this difficulty, we introduce an auxiliary scalar variable $\mu$ through the LMI constraint \eqref{suffiDARIH2GGi-a}, which is equivalent to the inequality $x_0^{\top}X^{-1}_{\gamma}x_0 \leq \mu$. As a consequence, minimizing $\mu$ under \eqref{suffiDARIH2GGi} is exactly equivalent to minimizing $x_0^{\top}X^{-1}_{\gamma}x_0$, but in a form that solvers can handle, leading to the following convex optimization problem.
   	\begin{mini}|s|
 		{X_{\gamma}, Y_{\gamma}, Z_{\gamma}, G_{\gamma}, \mu}{\mu}
 		{}{}
			\addConstraint{\eqref{suffiDARIH2GGi} \text{ holds.}}
			\label{optim3}
    \end{mini}	
%\begin{rem} For a given $x_0\in \R^{n}$ and $\gamma\in [0,1]$, it is important to note that the cost induced by the matrix $\widehat{K}_{\gamma}$, namely $J_{\gamma}(x_0,\widehat{K}_{\gamma})$, may differ from $x_0^{\top}\widecheck{P}_{\gamma}x_0$, where $\widecheck{P}_{\gamma}$ and $\widehat{K}_{\gamma}$ are solutions to \eqref{optim3}. Moreover, the convex optimization problem \eqref{optim3} depends on $x_0$, as is typically the case in near-optimal LQR problems \cite{VANW24}. Alternatively, the initial condition can be treated as an uncertain variable, allowing the optimization problem to be reformulated in a way that eliminates its dependence on $x_0$ \end{rem} \hfill $\Box$  . 
% The second approach, without using an initial condition $x_0\in \R^{n}$, transforms the problem into the minimization of a trace function \cite{Kleinman68}, leading to
%    	\begin{mini}|s|
%  		{X_{\gamma}, Y_{\gamma}, Z_{\gamma}, G_{\gamma}, W _{\gamma}}{\text {trace}(W_{\gamma})}
%  		{}{}
% 			\addConstraint{\eqref{suffiDARIH2GGi}}
% 			\label{optim3s}
%     \end{mini}	
% with the first constraint in \eqref{suffiDARIH2GGi} is replaced by $$
% \left[	\begin{array}{cc}
% 		W_{\gamma} &  \1 \\
% 		\1 & X_{\gamma} 
% 	\end{array}\right] \succ0.
% $$

We illustrate the outcomes of Theorem \ref{thmGCGXGG} on Example \ref{Ex: 1}.
\begin{revexample} 
%Consider the system and the cost in Example \ref{Ex: 1}.
Take the initial state $x_0=\begin{bmatrix}
    1&1
\end{bmatrix}^{\top}$ and the matrices $C=\left[\begin{smallmatrix}
    \sqrt{2}&0\\
    0& \sqrt{3}\\
    0&0
\end{smallmatrix}\right]$, $D=\left[\begin{smallmatrix}
    0\\
    0\\
    \sqrt{5}
\end{smallmatrix}\right]$ such that $C^{\top}D=\0$.  As desired $\widehat{K}_{\gamma}\in\mathcal{K}_1$ for all $\gamma\in [0,1]$. The evolution of the cost induced by the gain $\widehat{K}_{\gamma}$ solution of the convex optimization problem \eqref{optim3} as a function of $\gamma$ is shown in Fig.~\ref{fig:LQRgammaGFin}. We see that, as expected from \eqref{eq;Guarcost}, the cost $J_{\gamma}(x_0,\widehat{K}_{\gamma})$ (green plot) is always upper-bounded by $x_0^{\top}X^{-1}_{\gamma}x_0$ (blue dashed line), and lower-bounded by the optimal cost $x_0^{\top}P_{\gamma}x_0$ (red dashed line). The upper-bound prevents $J_{\gamma}(x_0,\widehat{K}_{\gamma})$ from becoming excessively large, ensuring a specified level of performance. Notably, for $\gamma\ge0.35$, the upper bound being close to the optimal cost forces $J_{\gamma}(x_0,\widehat{K}_{\gamma})$ to also remain near the optimal value. \hfill $\Box$

% First, we examine the relative error defined as
% \begin{equation}\label{eq:varepsilon}
% \varepsilon_{\gamma}(x_0,\widehat{K}_{\gamma}):=\tfrac{J_{\gamma}(x_0,\widehat{K}_{\gamma})-J_{\gamma}(x_0,K_{\gamma})}{J_{\gamma}(x_0,K_{\gamma})},
% \end{equation}
% which compares the cost induced by $\widehat{K}_{\gamma}$, namely $J_{\gamma}(x_0,\widehat{K}_{\gamma})$, with the optimal cost $J_{\gamma}(x_0,K_{\gamma})$, as a function of $\gamma$. Then we plot the norm of the error between $\widecheck{P}_{\gamma}$ and $P_{\gamma}$. We know that as $\gamma$ approaches one, the optimal solution 
%stabilizes with a common Lyapunov function. We recover this situation by imposing $Z_{\gamma}=X_{\gamma}$  (blue plot). Solving the optimization problem \eqref{optim3} with this constraint implies that $\widecheck{P}_{\gamma}$ matches the optimal solution $P_{\gamma}$ for $\gamma$ near one but may lead to a higher induced cost by $\widehat{K}_{\gamma}$ for smaller values of $\gamma$. By not imposing $Z_{\gamma}=X_{\gamma}$, and setting $G_{\gamma} = X_{\gamma}$, we achieve better results (red plot), recovering the optimal solution for $0.35 \leq \gamma \leq 1$ and providing a lower guaranteed cost while ensuring stability. Conservatism is further reduced by solving the optimization problem \eqref{optim3} without imposing constraints on $Z_{\gamma} $ and $G_{\gamma} $ as illustrated in the green plot.  \hfill $\Box$
\begin{figure}[t]
\centering
\includegraphics[width=0.42\textwidth]{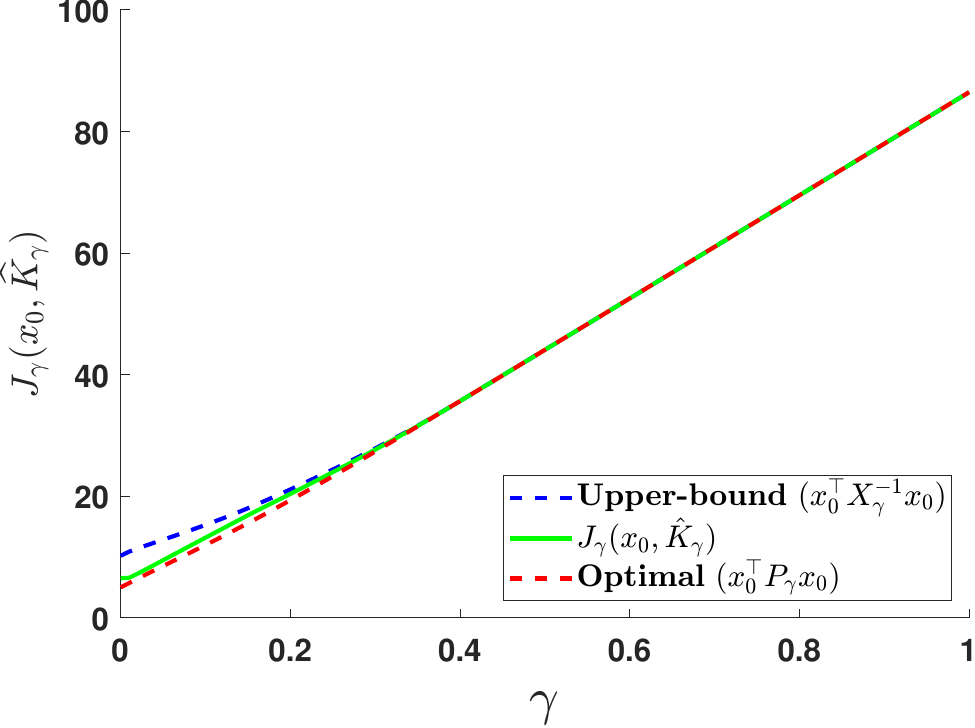}\vspace{-3mm}
    \caption{Cost $J_{\gamma}(x_0,\widehat{K}_{\gamma})$ as a function of $\gamma$ with $x_0=\begin{smallmatrix}
        [1&1]^{\top}
    \end{smallmatrix}$.}
    \label{fig:LQRgammaGFin}
    \vspace{-7mm}
\end{figure}
\end{revexample}

\subsection{LMI-based design with near-optimality bounds}\label{subsect:lmi-based-design-close-gain}
%Striving for stability while targeting optimality can be alternatively approached by answering the following  interesting question: does there exists a stabilizing control gain $\widebar{K}_{\gamma}$ that is close to the optimal gain $K_{\gamma}$? 

The design method presented in Section \ref{subsect:lmi-based-design-guaranteed-cost} ensures that the induced cost is less than a constant (that can be minimized using (\ref{optim3})) for a given initial condition.  In this section, we instead aim at designing a stabilizing gain, which we denote $\overline K_\gamma$, that minimizes the mismatch with the optimal gain $K_\gamma$. This problem can be formulated as  follows
    	\begin{mini}|s|
 		{\widebar{K}_{\gamma}\in {\cal K}_1}{\nu _{\gamma}}
 		{}{}
			\addConstraint{\| \widebar{K}_{\gamma}  - {K}_{\gamma}\| \leq \nu _{\gamma}}.
			\label{optimKbar}
	\end{mini}
This is equivalent to
    	\begin{mini}|s|
 		{\widebar{K}_{\gamma}, \widebar{P}_{\gamma}, \widebar{W}_{\gamma}}{\text{trace} (\widebar{W}_{\gamma})}
 		{}{}
			\addConstraint{ (\widebar{K}_{\gamma}  - {K}_{\gamma})^\top(\widebar{K}_{\gamma}  - {K}_{\gamma}) \preceq \widebar{W}_{\gamma},\,\widebar{P}_\gamma\succ 0}
			\addConstraint{(A+B\widebar{K}_{\gamma})^\top \widebar{P}_{\gamma}(A+B\widebar{K}_{\gamma}) - \widebar{P}_{\gamma} \prec 0}.
			\label{optimPbar}
	\end{mini}
The constraints in \eqref{optimPbar} are nonlinear, rendering this optimization problem non-convex. The next theorem gives sufficient LMI conditions to ensure the satisfaction of the constraints in \eqref{optimPbar}.
\begin{thm}\label{thm4}
Given any fixed $\gamma \in [0, 1]$ and consider the cost (\ref{eq:costd}) with $C^\top D = 0$. Suppose there exist $\widebar{L}_{\gamma}\in \R^{n\times n}$, $\widebar{Z}_{\gamma}\in \R^{n\times n}$ symmetric and $\widebar{S}_{\gamma}\in \R^{n\times n}$, $\widebar{T}_{\gamma}\in \R^{m\times n}$ such that
\begin{subequations}\label{suffiDARIH2GG}
\begin{align}
\left[ \begin{array}{cc} \widebar{S}_{\gamma}+\widebar{S}_{\gamma}^\top - \widebar{L}_{\gamma} &   (\bullet)^\top \\
\widebar{T}_{\gamma}-{K}_{\gamma}\widebar{S}_{\gamma}&  \1 \end{array}\right] \succeq 0 \label{suffiDARIH2GG-a}\\
\left[ \begin{array}{cc}  \widebar{S}_{\gamma}+\widebar{S}_{\gamma}^\top-\widebar{Z}_{\gamma} &  (\bullet)^\top \\
     A\widebar{S}_{\gamma}+B\widebar{T}_{\gamma} & \widebar{Z}_{\gamma} \end{array}\right] \succ 0. \label{suffiDARIH2GG-b}
\end{align}
\end{subequations}
Then $\widebar{K}_{\gamma}=\widebar{T}_{\gamma} \widebar{S}_{\gamma}^{-1}\in\mathcal{K}_1$ and $(\widebar{K}_{\gamma}  - {K}_{\gamma})^\top(\widebar{K}_{\gamma}  - {K}_{\gamma}) \preceq \widebar{L}_{\gamma}^{-1}$.  \mbox{}\hfill $\Box$
% satisfying the constraints of \eqref{optimPbar}
% Consider the discounted LQR problem  \eqref{eq:sysd}-\eqref{eq:costd} with $Q = C^\top C$, $R = D^\top D$,  $C^\top D = 0$, $x_0\in \R^{n}$ and  $\gamma \in [0, 1]$ fixed.  There exist a symmetric matrix $\widetilde{P}_{\gamma}\succ 0$ and a guaranteed cost gain $\widebar{K}_{\gamma}$ that belongs to ${\cal K}_1$
% satisfying the constraints of \eqref{optimPbar} if there exist symmetric matrices $\widebar{S}_{\gamma}\in \R^{n\times n}$, $\widebar{L}_{\gamma}\in \R^{n\times n}$ and a matrix $\widebar{T}_{\gamma}\in \R^{m\times n}$ solution to 
% \begin{equation}\label{suffiDARIH2GG}
% \begin{array}{c}
% \left[ \begin{array}{cc} \widebar{S}_{\gamma}+\widebar{S}_{\gamma}^\top - \widebar{L}_{\gamma} &   (\bullet)^\top \\
% \widebar{T}_{\gamma}-{K}_{\gamma}\widebar{S}_{\gamma}&  \1 \end{array}\right] \succ 0\\
% \left[ \begin{array}{cc}  \widebar{S}_{\gamma} &  (\bullet)^\top \\
%      A\widebar{S}_{\gamma}+B\widebar{T}_{\gamma} & \widebar{S}_{\gamma} \end{array}\right] \succ 0
% \end{array}
%  \end{equation}
%  Moreover, the near-optimal gain is given by $\widebar{K}_{\gamma} = \widebar{T}_{\gamma} \widebar{S}_{\gamma}^{-1}$. \mbox{}\hfill $\Box$
 \end{thm}
 
\begin{proof} The proof is derived by multiplying  \eqref{optimPbar} by $-1$ and following  a similar reasoning as in the proof of Theorem~\ref{thmGCGXGG}.\end{proof}  

%Theorem \ref{thm4} can then be used to design a stabilizing gain $\widebar{K}_\gamma$ while minimizing $\widebar{W}_{\gamma}$ or equivalently by maximizing $\widebar{L}_\gamma$ \jonathan{Not clear}, which provides a guaranteed upper-bound on $(\widebar{K}_{\gamma}  - {K}_{\gamma})^\top(\widebar{K}_{\gamma}  - {K}_{\gamma})$. 
Theorem \ref{thm4} can be used to design a stabilizing gain $\widebar{K}_\gamma$ with
a guaranteed upper-bound on $(\widebar{K}_{\gamma}  - {K}_{\gamma})^\top(\widebar{K}_{\gamma}  - {K}_{\gamma})$ by maximizing the trace of $\widebar{L}_\gamma$. This is achieved by solving the next convex optimization problem
\begin{maxi}|s|
{\widebar{S}_{\gamma}, \widebar{T}_{\gamma},  \widebar{L}_{\gamma}, \widebar{Z}_{\gamma}}{\text{trace} (\widebar{L}_{\gamma})}
{}{}
\addConstraint{\eqref{suffiDARIH2GG} \text{ holds.}}
			\label{optimSTbar}
\end{maxi}

\begin{revexample}
%We still consider the system and the cost in Example \ref{Ex: 1}. 
Fig.~\ref{fig:JgJb} compares the cost obtained with $\widehat{K}_{\gamma}$ and $\widebar{K}_{\gamma}$, solutions to \eqref{optim3} and \eqref{optimSTbar} respectively, for $\gamma \in [0, 1]$ and  $x_0=\begin{bmatrix}
    1&1
\end{bmatrix}^{\top}$. We introduce for this purpose the relative error of the cost induced by a gain $K\in \R^{m\times n}$ at an initial state $x_0\in \R^{n} \setminus\{0\}$, defined as
\begin{equation}\label{eq:varepsilon}
 \varepsilon_{\gamma}(x_0,K):=\frac{J_{\gamma}(x_0,K)-J_{\gamma}(x_0,K_{\gamma})}{J_{\gamma}(x_0,K_{\gamma})}, 
\end{equation}
    which is used to compare the cost induced by gain $K$, namely $J_{\gamma}(x_0,K)$, with the optimal cost $J_{\gamma}(x_0,K_{\gamma})$. We exclude  $x_0=0$ in the definition of $\varepsilon_\gamma$ as  $J_{\gamma}(x_0,K_{\gamma})=0$ in this case, which makes it ill-defined.  Hence, Fig.~\ref{fig:JgJb} presents the evolution of $\varepsilon_{\gamma}(x_0,\widehat{K}_{\gamma})$ and $\varepsilon_{\gamma}(x_0,\widebar{K}_{\gamma})$ as a function of $\gamma$. The results show that a stabilizing control gain can always be found by solving \eqref{optimSTbar}, with performance close to the optimal one. Notably, in this example, the gain $\widebar{K}_{\gamma}$ outperforms $\widehat{K}_{\gamma}$, achieving a relative error below $0.1\%$ for any $\gamma \in [0,1]$, whereas the cost obtained with $\widehat{K}_{\gamma}$ exhibits a relative error of up to $30\%$. Note that the cost $J_{\gamma}(x_0,\widebar{K}_{\gamma})$ differs from $J_{\gamma}(x_0,K_{\gamma})$ only for values of $\gamma$ where $K_{\gamma} \notin {\cal K}_1$. Even in these cases, the difference remains marginal.     Similar results were obtained by testing 16 initial conditions uniformly distributed within the square $[-1,1]\times[-1,1]$. This illustrates that, for this example, Theorem \ref{thm4} can be used to construct stabilizing state-feedback laws without significantly degrading the optimal cost. \hfill $\Box$
\begin{figure}[ht]
    \centering
    \includegraphics[width=0.35\textwidth]{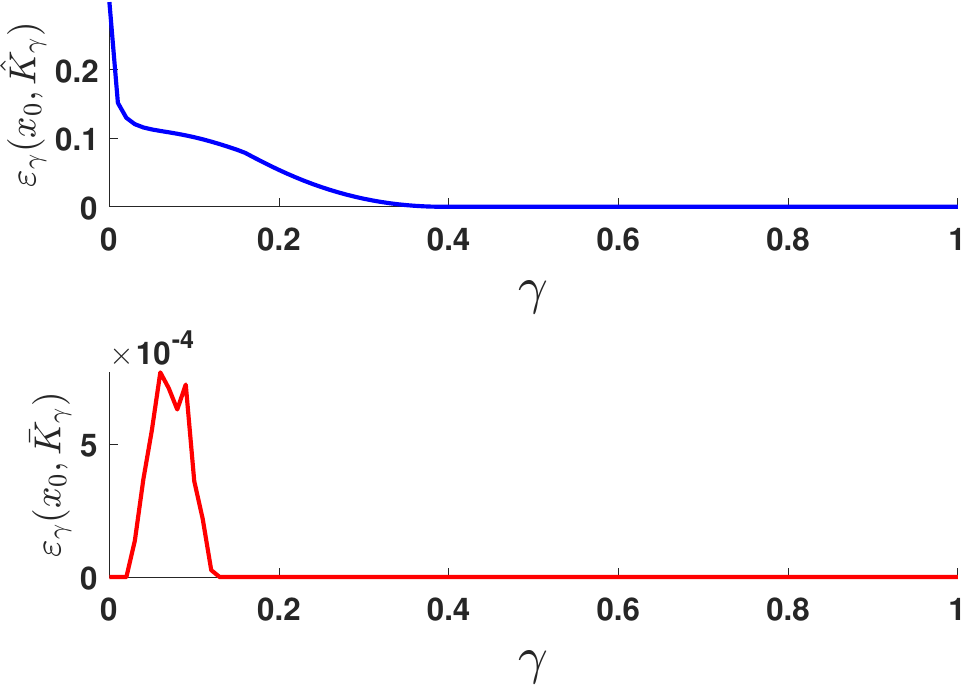}\vspace{-5mm}
    \caption{$\varepsilon_{\gamma}(x_0,\widehat{K}_{\gamma})$ and $\varepsilon_{\gamma}(x_0,\bar{K}_{\gamma})$ as functions of $\gamma$ for $x_0=\begin{smallmatrix}
        [1&1]^{\top}
    \end{smallmatrix}$.}
    \label{fig:JgJb}
\end{figure}
\end{revexample}

% As Theorems \ref{thmGCGXGG} and \ref{thm4} provide sufficient conditions, we may be interested in further  improving the performance of the gains obtained by any of the corresponding design methods. We propose for this purpose a modification of PI for this purpose.

To conclude this section, we finally present a novel variant of PI, which allows constructing a sequence of control gains in $\mathcal{K}_1$, with iteratively improved costs. Interestingly, this algorithm can be combined with the results presented so far in this section as we explain.

% \jonathan{
% \begin{rem}
%     Similar results were obtained by testing 16 initial conditions within the square $[-1,1]\times[-1,1]$. Indeed, simulations show that the relative error is reduced by a factor greater than 10 when using the control gain generated by \eqref{optimSTbar}. \hfill $\Box$
% \end{rem}}
\subsection{Near-optimal, recursively stabilizing PI}\label{subsect:pi}

%Finally, we present an alternative to the famous PI algorithm, which allows iteratively constructing near-optimal feedback gains that are guaranteed to stabilize system (\ref{eq:sysd}). 
%stability at any iteration while consistently improving the performance cost at each iteration.

{
\setlength{\textfloatsep}{0pt}
\begin{algorithm}
\caption{sPI}
\algsetup{
linenosize=\small,
linenodelimiter=.
}
\textbf{Input:} $K_0\in \mathcal{K}_1$, $\gamma\in [0,1]$, $A$, $B$ in \eqref{eq:sysd} and $Q$, $R$ in \eqref{eq:costd}\\
\textbf{Output:} $K_{\infty}$, $P_{\infty}$
\begin{algorithmic}[1]
\STATE \textbf{Initial evaluation step:}  \vspace{-5 pt}
    {\small\begin{equation}
        \label{sPI.1}
        \tag{sPI.1}
        \sqrt{\gamma}(A+BK_0)P_0\sqrt{\gamma}(A+BK_0)-P_0+K_0^{\top}RK_{0}+Q=0
    \end{equation}}
\vspace{-10 pt}
\FOR{$j\in \Z_{> 0}$}
\STATE \textbf{Improvement step:}\\
Select $\alpha_j\in(0,1]$ such that $K_j\in \mathcal{K}_1$ with
\begin{equation}
\begin{array}{rlll}
    \label{sPI.2}
    \tag{sPI.2}
    K_{j} \!\!&\!\! =\!\! &\!\!  \alpha _j(-\gamma  (R+\gamma  B^{\top}P_{j-1} B)^{-1}B^{\top}P_{j-1} A) \\
    & & + (1-\alpha _j)K_{j-1}
    \end{array}
\end{equation}\vspace{-10 pt}
\STATE \textbf{Evaluation step:}
    {\small\begin{equation}
        \label{sPI.3}
        \tag{sPI.3}
       \sqrt{\gamma}(A+BK_j)^{\top}P_j\sqrt{\gamma}(A+BK_j)-P_j+K_j^{\top}RK_{j}+Q=0
    \end{equation}} \vspace{-12 pt}
\ENDFOR
\RETURN $K_{\infty}\in \mathcal{K}_1$ and $P_{\infty}$.
\end{algorithmic}
\label{alg: sPI}
\end{algorithm}
}
The new algorithm called sPI is presented in Algorithm \ref{alg: sPI}. Given $\gamma \in [0,1]$ and an initial control gain $K_0\in \mathcal{K}_1$, sPI starts with the same initial evaluation step as in standard PI to generate  $P_0$ in  \eqref{sPI.1}. The differences with PI appear at the next iterations. At any iteration $j\in \Z_{>0}$, instead of solely aiming at improving the previous cost, sPI performs a linear combination between the controller gain that a standard PI improvement step would give and the previous controller gain, which is guaranteed to be in $\mathcal{K}_1$. We use parameter $\alpha_j$ to define this linear combination, whose value is selected to ensure that the newly obtained controller belongs to $\mathcal{K}_1$.  
%we select a parameter $\alpha_j\in [0,1]$ such that the updated control gain $K_j$, belonging to the convex hull of the previous gain and the classical PI gain as showed by \eqref{sPI.2}, is stabilizing, i.e., $K_j\in \mathcal{K}_1$. 
After this update of the gain, we find the classical evaluation step in \eqref{sPI.3} of PI to generate  matrix $P_j$. We further elaborate on the link between sPI and PI as well as other related algorithms in Remark \ref{rem:link-other-algos} below.

The next theorem ensures the recursive feasibility\footnote{ See \cite{granzotto2024robust} for further discussions on the recursive feasibility of the original PI algorithm.} of sPI in the sense that we can always find $\alpha_j$ and $K_j$ at the improvement step at any iteration. The theorem also recovers one of the main properties of PI: the obtained cost is improved at any iteration. Obviously, we cannot guarantee that the sequence of $P_j$ converges to $P_\infty$ as $K_\gamma$ does not a priori belong to $\mathcal{K}_1$, while the generated sequence of $K_j$ do.

\begin{thm} \label{th: sPI}
Given $\gamma\in [0,1]$ and $K_0 \in {\cal K}_1$, Algorithm \ref{alg: sPI} verifies the following properties.
\begin{enumerate}
    \item[(i)] \emph{(Recursive feasibility)} For any $j\in \Z_{>0}$, there always exist $\alpha_j \in (0,1]$ and $K_j\in \mathcal{K}_1$ verifying \eqref{sPI.2} and \eqref{sPI.3}. 
    \item[(ii)] \emph{(Recursive cost improvement)} For any $j\in \Z_{>0}$, we have $P_{j} \preceq P_{j-1}$, implying that the cost difference satisfies 
    \begin{equation} \label{eq: DeltaJ}
     J_{\gamma}(x_0, K_{j} ) - J_{\gamma}(x_0, K_{j-1}) \leq 0.
    \end{equation}
    for any $x_0\in \R^{n}$.\hfill $\Box$
\end{enumerate}\end{thm}

\begin{proof}
(i) Let $\gamma\in [0,1]$. As $K_0\in\mathcal{K}_1$, there always exists $\alpha_1\in(0,1]$ sufficiently small such that $\rho(A+BK_1)<1$ with $K_1$ as in (\ref{sPI.2}) by  continuity of the eigenvalues of $A+BK_1$ with respect to $\alpha_1$. The same reasoning applies to any $j\in\Z_{>0}$ thereby ensuring that item (i) of Theorem \ref{th: sPI} holds.

(ii) Let $\gamma\in[0,1]$, $x_0\in\R^{n}$, $j\in\Z_{>0}$ and  $\Delta K_j :=K_{j}-K_{j-1}$. From (\ref{sPI.2}), we have $\Delta K_j = -\alpha_j M_{j-1}$
with  
\begin{equation}
M_{j-1} =   \gamma ( R+\gamma B^\top P_{j-1}B)^{-1} B^\top P_{j-1} A + K_{j-1}.  \end{equation}
%To show that $P_{j+1} < P_{j}$, $\forall j=0, 1, 2, \ldots$, we follow similar steps as in the proof of \cite[Lemma 2]{Agazzi2020} to arrive at, for all $x$:
%\begin{multline*}
%x^\top P_{j+1}x - x^\top P_{j}x = 2 x^\top\Delta K^\top \big( (\gamma B^\top P_{j+1}B+R)K_j  \\ 
%+ \gamma B^\top P_{j+1} A  \big)x+ x^\top\Delta K^\top (\gamma B^\top P_{j+1}B+R) \Delta K x
%\end{multline*}
%By substituting $\Delta K= -a_j M_j$,
%and using the relation
%$$
%\gamma B^\top P_{j+1} A =  ( B^\top P_{j+1}B+R) M_j -  ( B^\top P_{j+1}B+R) K_j
%$$
%we have
%{\small \begin{multline*}
%x^\top P_{j+1}x - x^\top P_{j}x = 2 
%a_j(a_j-2) x^\top M_j^\top (\gamma B^\top P_{j+1}B+R) M_j x
%\end{multline*}}
%and hence $P_{j+1} -  P_{j} \leq 0$ by the fact that $a_j \in [0, 1]$ combined with the positive definiteness of $ (\gamma B^\top P_{j+1}B+R) $. One then ends the proof and concludes that $\Delta J$ remains negative. 
We aim at showing that  
\begin{equation}
\Delta J_{\gamma}(x_0,K_{j},K_{j-1}):=J_{\gamma}(x_0, K_{j} ) - J_{\gamma}(x_0, K_{j-1})\leq 0.
\end{equation}
We note (\ref{sPI.3}) 
ensures that $K_{j-1} \in {\cal K}_{\gamma}$ and  $J_{\gamma}(x_0, K_{j-1} ) = x_0^{\top} P_{j-1}x_0=\text{trace}(x_0^{\top} P_{j-1}x_0)$. By cyclicity of the trace, we have 
\begin{equation*}\label{eq:proof-sPI-J-Kj-1}
\begin{array}{rlll}    
J_{\gamma}(x_0, K_{j-1} ) & = &  \text{trace}(P_{j-1}x_0x_0^{\top}) \\
&=& \text{trace}\Big(P_{j-1} \big( S_{j-1}-\sqrt{\gamma}(A+BK_{j-1})\\&& S_{j-1}\sqrt{\gamma}(A+BK_{j-1})^\top\big) \Big),
\end{array}
\end{equation*} 
with $S_{j-1} \succeq 0$ solution to  
$\sqrt{\gamma}(A+BK_{j-1}) S_{j-1}\sqrt{\gamma}(A+BK_{j-1})^\top  - S_{j-1} +x_0x_0^{\top} = 0$. Using once again the cyclicity and the linearity of the trace function we have
\begin{align}
    J_{\gamma}(x_0, K_{j-1} ) = & \, \text{trace}\big(S_{j-1}P_{j-1}\big)\nonumber\\
    & - \text{trace}\big(S_{j-1}\sqrt{\gamma}(A+BK_{j-1})^{\top}P_{j-1}\times\nonumber\\
    & \sqrt{\gamma}(A+BK_{j-1})\big)\nonumber
\end{align} 
As $P_{j-1}=\sqrt{\gamma}(A+BK_{j-1})^{\top}P_{j-1}\sqrt{\gamma}(A+BK_{j-1})+K_{j-1}^{\top}RK_{j-1}+Q$ by (\ref{sPI.3}), we derive
\begin{align}
    J_{\gamma}(x_0, K_{j-1} ) = & \, \text{trace}\big(
    S_{j-1}\sqrt{\gamma}(A+BK_{j-1})^{\top}P_{j-1}\times \nonumber\\
    &\sqrt{\gamma}(A+BK_{j-1})\big)\nonumber\\
    &+\text{trace}\big(S_{j-1}(Q+K_{j-1}^{\top}RK_{j-1})\big)\nonumber\\
    &- \text{trace}\big(S_{j-1}\sqrt{\gamma}(A+BK_{j-1})^{\top}P_{j-1}\times\nonumber\\
    & \sqrt{\gamma}(A+BK_{j-1})\big)\nonumber\\
    =&\,\text{trace}\big(S_{j-1}(Q+K_{j-1}^{\top}RK_{j-1})\big). \label{eq:proof-sPI-J-Kj-1}
\end{align} 

We similarly derive from (\ref{sPI.3}) that $K_{j} \in {\cal K}_{\gamma}$ and  the associated cost  $J_{\gamma}(x_0, K_{j} ) = x_0^{\top} P_{j}x_0$ verifies
\begin{align}
J_{\gamma}(x_0, K_{j} ) & =   \text{trace}(P_{j}x_0x_0^{\top}) \nonumber\\
 & =   \text{trace}(S_{j} (Q + K_{j}^{\top}RK_{j} ) ) \label{eq:proof-sPI-J-Kj}
\end{align}
with $S_{j} \succeq 0$ solution to 
$   \sqrt{\gamma}(A+BK_{j})S_{j}\sqrt{\gamma}(A+BK_{j})^\top  - S_{j} +x_0x_0^{\top} = 0 
$. 
By (\ref{eq:proof-sPI-J-Kj-1}) and (\ref{eq:proof-sPI-J-Kj}), we derive like in  \cite[(11)]{Makila87} that
 \begin{multline*}
\Delta J_{\gamma}(x_0,K_{j},K_{j-1}) = \text{trace}\Big(2 \Delta K_{j}^\top \big( (\gamma B^\top P_{j-1}B+R) \times  \\ 
K_{j-1} S_{j} + \gamma B^\top P_{j-1} A S_{j} \big)+\Delta K_j^\top (\gamma B^\top P_{j-1}B+R) \Delta K_j S_{j} \Big) 
\end{multline*}
By substituting $\Delta K_j = -\alpha_j M_{j-1}$,
and using the relation
$$
\gamma B^\top P_{j-1} A =  ( \gamma B^\top P_{j}B+R) M_{j-1} -  ( \gamma B^\top P_{j-1}B+R) K_{j-1}
$$
we arrive at the next expression
%\begin{multline*}
%\Delta J = Tr\Big(-2 a_j M_j^\top (\gamma B^\top P_{j+1}B+R)M_jS_{j+2}  \\
% + a_j^2 M_j^\top (\gamma B^\top P_{j+1}B+R) M_j S_{j+2} \Big)
%\end{multline*}
\begin{align*}
\MoveEqLeft \Delta J_{\gamma}(x_0,K_{j},K_{j-1}) & \\ &{}=\text{trace}\big(-2 \alpha_j M_{j-1}^\top (\gamma B^\top P_{j-1}B+R)M_{j-1}S_{j}  \\
 & \quad {}+ \alpha_j^2 M_{j-1}^\top (\gamma B^\top P_{j-1}B+R) M_{j-1} S_{j} \big) \\
& {}= \alpha_j(\alpha_j-2) \text{trace}\big( M_{j-1}^\top (\gamma B^\top P_{j-1}B+R) M_{j-1} S_{j} \big).
\end{align*}
We deduce that $\Delta J_{\gamma}(x_0,K_{j},K_{j-1}) \leq 0$ as $\alpha_j \in (0, 1]$, together with the symmetry and positive semi-definiteness of $S_{j}$ and $\gamma B^\top P_{j-1}B+R$. We have proved (\ref{eq: DeltaJ}) for any arbitrary $j\in\Z_{>0}$ and $x_0\in\R^{n}$.  Consequently, $P_{j}\preceq P_{j-1}$,  which concludes the proof.
\end{proof}

Theorem \ref{th: sPI} guarantees that the controller gain obtained at each iteration stabilizes (\ref{eq:sysd}), and that its cost cannot be bigger than the cost of the gain obtained at the previous iteration. A relevant way to initialize sPI is to take $K_0\in\{\widehat{K}_\gamma,\overline{K}_\gamma\}$ with $\widehat{K}_\gamma,\overline{K}_\gamma$ solution of \eqref{optim3} and \eqref{optimSTbar}, respectively. In this way, sPI can only provide ``better'' controller gains belonging to $\mathcal{K}_1$.

\begin{rem}\label{rem:link-other-algos} We briefly elaborate on the link between sPI and existing PI(-like) algorithms. First, when, at any iteration $j\in \Z_{\ge0}$ we can take   $\alpha_j=1$, we recover the standard iteration of the original PI  \cite{Hewer71}. Second, 
sPI shares similarities with policy gradient methods. Indeed, by adapting \cite[(7)]{Fazel_2018} to the discounted setting, \eqref{sPI.2} can be interpreted as an update of a Gauss-Newton method
\begin{equation*}
K_{j+1}=(1-2\eta)K_j+2\eta(-\gamma(R+\gamma B^{\top}RB)^{-1}B^{\top}P_jA)
\end{equation*} 
with $2\eta=\alpha_j$. Compared to \cite[(7)]{Fazel_2018}, in sPI: (i) the weights $\alpha_j$ are allowed to change at any iteration instead of being fixed; (ii) we impose $K_j$ to belong to $\mathcal{K}_1$ at each iteration to recursively obtain a stabilizing gain. Our approach also differs  from \cite{zhao2021}, which relies on \eqref{eq:LyapKRK} and incorporates an updating mechanism for the discount factor to ensure that the gradient descent process yields a stabilizing policy, while sPI works for a fixed discount factor. 
\hfill $\Box$
\end{rem}

\begin{rem}\label{rem:stopping-criterion}
A relevant way to implement $\alpha_j$ is to maximize its value at each iteration. Then, a possible stopping criterion for the algorithm when $K_{\gamma}\notin\mathcal{K}_1$ is, given a  threshold $\epsilon>0$, to stop iterating when   there exists $j\in\Z_{>0}$ such that $\alpha_{j}<\epsilon$.~\hfill $\Box$
\end{rem}

\begin{revexample}
Let  $\gamma=0.1$, we have that $K_{\gamma}\notin\mathcal{K}_1$, see Example \ref{Ex: 1}. We begin with the initial stabilizing gain $K_0:=\widehat{K}_{\gamma}\in \mathcal{K}_1$, where $\widehat{K}_{\gamma}$ is the solution to \eqref{optim3} with $x_0=\begin{bmatrix}
    1&1
\end{bmatrix}^{\top}$. At any iteration $j\in\Z_{>0}$,  we discretize the interval $[0,1]$ using a grid of length $10^{-5}$ and we denote by $\widebar{\alpha}_j$  the maximal value $\alpha_j$ in this  grid such that the improvement step in Algorithm \ref{alg: sPI} is feasible. To illustrate sPI  over a large number of iterations, we select $\alpha_j=\tfrac{\overline{\alpha}_j}{10}$. %At each step, we test the stability of the closed-loop system $A+BK_j$, where $K_{j}$ is given by \eqref{alg: sPI2}. The maximum value of $\alpha_j$ that satisfies the stability constraint is denoted by $\bar{\alpha}_j$. 
Since the optimal control gain is not stabilizing for this value of $\gamma$, we set the stopping criterion to $\alpha_j<10^{-5}$, see Remark \ref{rem:stopping-criterion}. The evolution of the eigenvalues of $A+BK_{j}$ over the iterations $j$ is presented in Fig.~\ref{fig:sPI_eigen} illustrating the fact that they lie inside the unit circle at any iteration as desired. 

\begin{figure}[ht]
    \centering
    \includegraphics[width=0.35\textwidth]{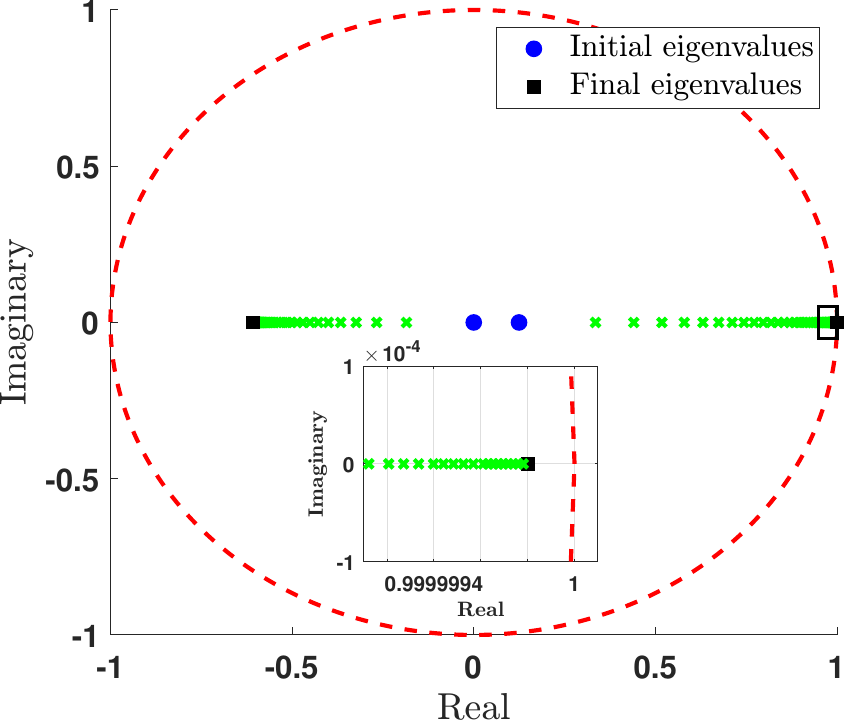}
    \caption{Unit disk (dashed red curve) and eigenvalues of $A+BK_{j}$ in the complex plane obtained with sPI (green crosses), with a zoomed-in view near the region where the real parts are close to 1, with $\gamma=0.1$ and $K_0=\widehat{K}_{\gamma}=\begin{smallmatrix}
        [-0.0081&-0.1409]
    \end{smallmatrix}$.}
    \label{fig:sPI_eigen}
\end{figure}
For any $x_0\in \R^{n}$ and any iteration $j\in\Z_{\ge0}$, the cost induced by $K_j$ is given by $x_0^\top P_jx_0$. Hence, the closer $P_j$ is to $P_{\gamma}$ as in (\ref{eq:optimal-value-function}), the closer the cost induced by $K_j$ is to the one induced by $K_\gamma$ for any $x_0\in \R^{n}$. Fig.~\ref{fig:sPI_cost} shows  $\norm{P_j-P_{\gamma}}$ over the iterations $j$. We observe that $\norm{P_j-P_{\gamma}}$ decreases with $j$ consistently with item (ii) of Theorem \ref{th: sPI}. Although $\norm{P_j-P_{\gamma}}$ approaches zero after $40$ iterations, the fact that $K_{\gamma}\notin \mathcal{K}_1$, implies that a zero error cannot be asymptotically reached. To  illustrate this, we provide a zoomed-in view for $j\ge100$, revealing a residual error on the order of $10^{-3}$.

\begin{figure}[ht]
    \centering
    \includegraphics[width=0.32\textwidth]{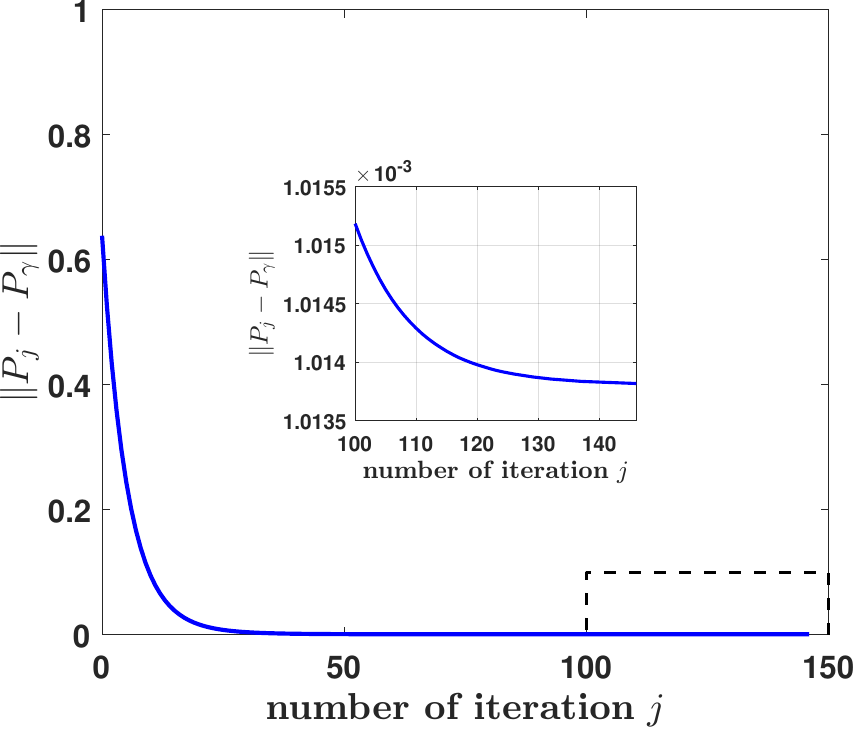}
    \caption{Norm  of $P_j - P_{\gamma}$ as a function of the number of iterations $j$ for $\gamma = 0.1$ and $K_0 = \begin{smallmatrix} [-0.0081 & -0.1409] \end{smallmatrix}$, with a zoomed-in view for $j > 100$.}
    \label{fig:sPI_cost}
\end{figure}

\end{revexample}
% \begin{rem}
% Note that the control gain $\widehat{K}_{\gamma}$ and $\widebar{K}_{\gamma}$ solutions of the optimization problems \eqref{optim3} and \eqref{optimKbar}, respectively, can be used as initial gain $K_0\in \mathcal{K}_{1}$. \hfill $\Box$    
% \end{rem}

%%%%%%%%%%%%%%%%%%%%%
\section{Conclusion}\label{sect:conclusions}
%%%%%%%%%%%%%%%%%%%%%%%%
We have presented novel stability (necessary and) sufficient conditions involving the optimal value function for the discounted LQR. % formulated as matrix inequalities, providing improved and even exact stability criteria for determining the range of $\gamma$ values that ensure the stability of the optimal closed-loop system.
Moreover, when the optimal controller fails to stabilize the system, we have given  methods to design near-optimal stabilizing state feedback laws both by solving convex optimization problems with LMI-based constraints and with a variant of policy iteration. The considered numerical examples show that the proposed approach may allow designing stabilizing controller without significantly degrading the optimal performance.

This work allows for the exploration of various research directions. First,  we plan to quantify the near-optimality bound in terms of  mismatch between the cost induced by the near-optimal controllers and the optimal value function $V^{\star}_\gamma$. Second, it would be interesting to investigate more general near-optimal controllers than static state-feedback laws as done in this work, as pertinently brought to our attention by Leonid Mirkin. Third, it would be  relevant to exploit the findings of Section \ref{sec: Lyap analysis} in a learning context similar to \cite{Tyler22, LAI2023}.

\bibliographystyle{IEEEtran}
\bibliography{H2_Discount}

% \newpage

% {
% \setlength{\textfloatsep}{0pt}
% \begin{algorithm}
% \caption{sPI2}
% \algsetup{
% linenosize=\small,
% linenodelimiter=.
% }
% \textbf{Input:} $K_0\in \mathcal{K}_1$, $\gamma\in [0,1]$, $A$, $B$ in \eqref{eq:sysd} and $Q$, $R$ in \eqref{eq:costd}\\
% \textbf{Output:} $K_{\infty}$, $P_{\infty}$
% \begin{algorithmic}[1]
% \STATE \textbf{Initial evaluation step:}  \vspace{-5 pt}
%     {\small\begin{equation}
%         \label{sPI2.1}
%         \tag{sPI.1}
%         \sqrt{\gamma}(A+BK_0)P_0\sqrt{\gamma}(A+BK_0)-P_0+K_0^{\top}RK_{0}+Q=0
%     \end{equation}}
% \vspace{-10 pt}
% \FOR{$j\in \Z_{\geq 0}$}
% \STATE \textbf{Improvement step:}   \begin{maxi*}|l|
% {\alpha\in(0,1]}{\alpha}
% {}{\alpha_j:={}}
% \addConstraint{K^+(P_j,K_j,\alpha)\in \mathcal{K}_1}
% \end{maxi*}
% \begin{equation}
% \begin{array}{rlll}
%     \label{sPI2.2}
%     \tag{sPI.2}
%     K_{j+1} := K^+(P_j,K_,\alpha_j)
%     \end{array}
% \end{equation}
% \STATE \textbf{Evaluation step:}\vspace{-5 pt}
%     {\small\begin{equation}
%         \label{sPI2.3}
%         \tag{sPI.3}
%        \sqrt{\gamma}(A+BK_j)P_j\sqrt{\gamma}(A+BK_j)-P_j+K_j^{\top}RK_{j}+Q=0
%     \end{equation}} \vspace{-15 pt}
% \ENDFOR
% \RETURN $K_{\infty}\in \mathcal{K}_1$ and $P_{\infty}$.
% \end{algorithmic}
% \label{alg: sPI2}
% \end{algorithm}
% }

% $$ K^+(P,K,\alpha):=  \alpha(-\gamma  (R+\gamma  B^{\top}P B)^{-1}B^{\top}P A) + (1-\alpha)K $$
\end{document}